\documentclass[12pt]{amsart}

\usepackage{amssymb}
\usepackage{float}

\newtheorem{thm}{Theorem}[section]
\newtheorem{lem}[thm]{Lemma}

\theoremstyle{remark}

\newtheorem{remark}[thm]{Remark}

\newcommand{\Aff}{\mathrm{Aff}}  

\begin{document}

\title{Cyclic Sieving for Strong Dichotomy Enumeration}

\author{Octavio A.  Agustín-Aquino}
\address{Instituto de Física y Matemáticas, Universidad Tecnológica de la Mixteca, Huajuapan de León, Oaxaca, México}
\email{octavioalberto@mixteco.utm.mx}
\urladdr{www.utm.mx/$\sim$octavioalberto} 

\begin{abstract}
Agustín-Aquino \cite{oA18} solved, in terms of the table of marks of $\Aff(\mathbb{Z}/2k\mathbb{Z})$,
the problem of enumerating the classes of bicolour self-complementary and rigid patterns in $\mathbb{Z}/2k\mathbb{Z}$ (also
known as \emph{strong dichotomy classes}). In particular, the rigid pattern-inventory polynomial appeared,
for odd $k$, to yield the number of strong classes with negative sign when evaluated in $-1$, and it was conjectured
that this is true for $k$ a power of an odd prime. Here we prove the conjecture is true for $k$ odd in general.
\end{abstract}

 \maketitle

\section{Introduction}

The enumeration of bicolour patterns in music, in particular within $\mathbb{Z}/n\mathbb{Z}$, is important
for the study of scales, chords, rhythmic patterns \cite{hF91}, rhythmic canons \cite{hF04}, and counterpoint \cite{AJM15}.
In this last regard, we take $n=2k$, since the set of pitches in $2k$-tone equal temperament is divided into
sets of consonances and dissonances of the same cardinality. In
particular, patterns that are self-complementary and rigid (which are termed \emph{strong dichotomies}) are
of particular interest, for they can be seen as generalisations of the standard notion of consonance and dissonance from the Renaissance and onwards,
namely
\[
(K/D) = (\{0,3,4,7,8,9\},\{1,2,5,6,10,11\}).
\]

More specifically, and because of musical reasons, we consider patterns to be equivalent under the action of
the affine group $\Aff(\mathbb{Z}/2k\mathbb{Z})$. In \cite{oA18} the enumeration polynomial for the strong
dichotomy patterns was obtained in terms of the table of marks, and it was observed that the polynomial for the rigid
pattern classes $Q_{\mathrm{rig}}(x)$, evaluated at $-1$, appears to satisfy
\begin{equation}\label{E:Conjecture}
Q_{\mathrm{rig}}(-1)=-s(2k),
\end{equation}
where $s(2k)$ is the number of strong dichotomy classes. Thus, in terms of the cyclic sieving phenomenon \cite{RSW04},
it was conjectured that this always holds when $k$ is a power of an odd prime (because $(\mathbb{Z}/2k\mathbb{Z})^{\times}$ is cyclic in that case). We prove that it actually holds for every odd positive integer $k$. 

In Section 2 we gather some definitions, facts and notation in relation to the Möbius function over posets. In Section
3 we introduce the main algebraic combinatorial constructions, while in Section 4 we transform the left-hand side
$Q_{\mathrm{rig}}(-1)$. In Section 5 we examine the right-hand side $s(2k)$. Finally, in Section 6 we prove that
the two expressions agree up to the required sign. In Section 7 we provide some values
of the number of strong dichotomies using the algorithm of the main theorem implemented in GAP.

\section{Preliminaries on the Möbius Function for Posets}
Although the following facts are standard (see \cite[Section 5.2]{mA07}), we include them for
the sake of completeness and the use of certain notation. Given a poset $(P,\leq)$, its \emph{incidence algebra} is
\[
\mathbb{A}(P):=\{\alpha:P\times P \to \mathbb{C}: \alpha(x,y)=0\text{ unless }x\leq y\}.
\]

If we define a multiplication in $\mathbb{A}(P)$ by
\[
(\alpha\beta)(x,y) = \sum_{z\in P}\alpha(x,z)\beta(z,y)
\]
then the function
\[
I(x,y):=[x=y]
\]
acts as the identity (here $[S]$ is the \emph{Iverson bracket}, which yields $1$ if $S$ is true
and $0$ otherwise). We also have the \emph{zeta function}
\[
\zeta(x,y):=[x\leq y]
\]
and the \emph{Möbius function}
\[
\mu(x,y):=\begin{cases}
0, &x\not\leq y,\\
1, &x=y,\\
-\sum_{x\leq z <y}\mu(x,z), &\text{otherwise}.
\end{cases}
\]

The following holds
\begin{align}
\mu\zeta(x,y) &= \sum_{z\in P}\mu(x,z)\zeta(z,y)\notag\\
&= \sum_{z\in P}\mu(x,z)[z\leq y]\notag\\
&= \sum_{x\leq z\leq y}\mu(x,z),\label{E:Mobius}\\
&=\begin{cases}
0, & x\not\leq y,\\
\mu(x,x), &x=y,\\
\sum_{x\leq w < y}\mu(x,w)+\mu(x,y)=-\mu(x,y)+\mu(x,y)=0, &x<y,\end{cases}\notag\\
&=[x=y] = I(x,y).\notag
\end{align}

Also
\begin{align*}
\zeta\mu(x,y) &= \sum_{z\in P}\zeta(x,z)\mu(z,y)\\
&= \sum_{z\in P}[x\leq z]\mu(z,y) = \sum_{x\leq z\leq y}\mu(z,y)=I(x,y)
\end{align*}
(see \cite[Proposition 3.6.2]{rS97}). Note that in general the product
\[
\alpha\zeta(x,y) = \sum_{x\leq z\leq y}\alpha(x,z)
\]
is a cumulative sum of $\alpha$, and that Möbius inversion amounts to the fact that $\alpha\zeta = \beta$ is
equivalent to $\alpha = \beta\mu$. In particular, if $\alpha\zeta = \beta\zeta$, and given that
$\zeta$ is invertible, then we have $\alpha=\beta$.

\section{The Rigid Pattern-Inventory Polynomial}

The pattern-inventory polynomial for bicolour patterns within $S=\mathbb{Z}/2k\mathbb{Z}$ under the action of
\[
 G=\Aff(\mathbb{Z}/2k\mathbb{Z}):=(\mathbb{Z}/2k\mathbb{Z}) \rtimes (\mathbb{Z}/2k\mathbb{Z})^{\times}
\]
that are fixed by the trivial subgroup only is
\[
Q(z_{1},\ldots,z_{2k}) = \sum_{j=1}^{N}b_{N,j}P_{j}(z_{1},\ldots,z_{2k}),
\]
where $G_{i}$ are the subgroups of $\Aff(\mathbb{Z}/2k\mathbb{Z})$ and are indexed such that
\[
|G_{1}|\geq \cdots \geq |G_{N}|=1,
\]
the coefficients $b_{i,j}$ come from the inverse of the tables of marks matrix and
\[
P_{j}(z_{1},\ldots,z_{2k})=\prod_{d=1}^{2k}z_{d}^{\sum_{i=1}^{\ell}[|O_{G_{i}}|=d]}
\]
are the orbit index monomials, where $\ell$ is the number of orbits of the subgroup $G_{i}$ (see \cite{oA18} for
details).

Upon the substitution $z_{r}\mapsto 1+x^{r}$ (see \cite[Section 6.2]{mA07}), we obtain the inventory polynomial of the set of rigid subsets
\[
\mathcal{R}=\{A\subseteq \mathbb{Z}/2k\mathbb{Z}:G_{A}=1\}
\]
which is
\[
Q_{\mathrm{rig}}(x) = \sum_{A\in \mathcal{R}/G} x^{|A|}.
\]

Since each $A$ has an orbit of $|G|$ distinct sets given its rigidity, then
\[
Q_{\mathrm{rig}}(x) = \frac{1}{|G|}\sum_{A\subseteq \mathbb{Z}/2k\mathbb{Z}}[G_{A}=1]x^{|A|}.
\]

Now, using \eqref{E:Mobius} applied to the subgroup poset of $G$, we have
\begin{equation}\label{E:IdMobius}
 \sum_{1\leq H \leq G_{A}} \mu(1,H) = [1=G_{A}],
\end{equation}
thus
\begin{align*}
Q_{\mathrm{rig}}(x) &= \frac{1}{|G|}\sum_{A\subseteq \mathbb{Z}/2k\mathbb{Z}}\sum_{1\leq H \leq G_{A}} \mu(1,H)x^{|A|}\\
&=\frac{1}{|G|}\sum_{1\leq H \leq G}\sum_{A:HA=A} \mu(1,H)x^{|A|}\\
&=\frac{1}{|G|}\sum_{1\leq H \leq G}\mu(1,H)\sum_{A:HA=A}x^{|A|}
\end{align*}
because summing over the subgroups of the stabilizer of a subset is the same as summing over all the subgroups that
leave a set invariant.

\section{The Rigid Side of the Equation}

In preparation for the proof we evaluate the polynomial $Q_{\mathrm{rig}}$ at $-1$, so we can see how the Möbius expansion sieves the $H$-orbits by parity,
\[
Q_{\mathrm{rig}}(-1) = \frac{1}{|G|}\sum_{1\leq H \leq G}\mu(1,H)\sum_{A:HA=A}(-1)^{|A|}.
\]

Every subset that is invariant under a subgroup $H$ has to be the union of the orbits induced by such a subgroup,
thus
\[
 \sum_{A:HA=A}(-1)^{|A|} = \prod_{O\in S/H}(1+(-1)^{|O|}),
\]
and this is not zero only when all the orbits are of even cardinality. If we denote with $\mathcal{E}$ the set of
subgroups such that all of their orbits are of even cardinality, we can write
\begin{equation}\label{E:WhiteSide}
Q_{\mathrm{rig}}(-1) = \frac{1}{|G|}\sum_{1\leq H \leq G,H\in\mathcal{E}}\mu(1,H)2^{|S/H|}
\end{equation}

\section{The Other Side of the Equation}

On the other hand, define (as in \cite{oA12})
\[
M_{q} = \left\{D\in\binom{S}{k}:qD=\complement D\right\},
\]
which is the set of dichotomies of $S$ with quasipolarity $q$. Let $\mathcal{Q}(G)$ denote the set of quasipolarities, i.e.,
all $q\in G$ such that they are derangements and $q^{2}=1$.

\begin{remark}
The quasipolarity of a rigid dichotomy $D$ is unique, for if
$qD=\complement D$ and $rD=\complement D$, then
$r^{-1}qD=D$, so $r^{-1}q\in G_D=1$, and hence $q=r$.
\end{remark}

The number of classes of strong dichotomies is
\[
 s(2k) := \frac{1}{|G|}\sum_{q\in \mathcal{Q}(G)}\sum_{D\in M_{q}}[G_{D}=1]
\]
and, using \eqref{E:IdMobius} again,
\begin{align*}
s(2k) &= \frac{1}{|G|}\sum_{q\in \mathcal{Q}(G)}\sum_{D\in M_{q}}\sum_{1\leq H\leq G_{D}} \mu(1,H)\\
&= \frac{1}{|G|}\sum_{q\in \mathcal{Q}(G)}\sum_{1\leq H\leq G} \sum_{D\in M_{q},HD=D} \mu(1,H)\\
&= \frac{1}{|G|}\sum_{q\in \mathcal{Q}(G)}\sum_{1\leq H\leq G}\mu(1,H)|\{D\in M_{q}:H\leq G_{D}\}|\\
&= \frac{1}{|G|}\sum_{q\in \mathcal{Q}(G)}\sum_{1\leq H\leq G}\mu(1,H)|M_{q}^{H}|,
\end{align*}
where
\[
M_{q}^{H}:=\{D\in M_{q}:H\leq G_{D}\}.
\]

Define
\[
K_{0}:=\{e^{u}.v\in G:u\equiv 0\pmod{2}\}
\]
i.e., the set of affine symmetries with even translational part.

\begin{lem}\label{L:CaracKernel}
Let $k$ be an odd positive integer. A subgroup $H\leq G$ has only even orbits if, and only if, $H\not\leq K_{0}$.
In other words, $H\in\mathcal{E}$ if, and only if, $H\not\leq K_{0}$.
\end{lem}
\begin{proof}
Suppose $H\leq G$ has only even orbits. Then some orbit has both even and odd elements of $S$, for otherwise the
even part of $S$ would be a union of disjoint sets of even cardinality, and thus $k$ would be even, which is a contradiction.
Hence at least one $e^{u}.v\in H$ has to swap parities. Given that $v$ is odd (for $\gcd(v,2k)=1$), then $u$ has to be odd, which means $H\not\leq K_{0}$.

Now suppose $H\leq G$ has at least one odd orbit. Let $K=H\cap K_{0}$. Suppose there exists $\tau \in H\setminus K_{0}$. Then any orbit $Hx$ can be decomposed as $Kx\sqcup \tau Kx$ and, since $|Kx|=|\tau Kx|$, this means all the orbits are of even cardinality. The contradiction implies $H\subseteq K_{0}$.
\end{proof}

\begin{remark}\label{R:Empty}
Note that $M_{q}^{H}=\emptyset$ for $H\not\leq K_{0}$. Indeed: since every $H$-orbit is of even cardinality, then the dichotomy has to be obtained as a union of orbits. But $|D|=k$ is odd, thus no dichotomy can be invariant under $H$.
\end{remark}

\begin{remark}
The quasipolarities $p=e^{u}.v$ are not members of $K_{0}$ when $k$ is odd. Indeed, all the orbits of $p$ are of length $2$,
thus the even (or odd) part of $\mathbb{Z}/2k\mathbb{Z}$ cannot be a union of them. This implies that $x$ and $p(x)$ are
of opposite parity for at least one $x$, which implies that $u$ is odd.
\end{remark}

\begin{lem}\label{L:Orbits} For $k$ odd, $H\leq K_{0}$ and $q\in\mathcal{Q}(G)$ we have
\[
 |M_{q}^{H}|=2^{|S/L|}
\]
where $L=\langle H,q\rangle$.
\end{lem}
\begin{proof}
From the proof of Lemma \ref{L:CaracKernel} we can readily see that all orbits of $L$ can be decomposed as $Kx\sqcup qKx$ with $K=L\cap K_{0}$, and in particular we can pick either $Kx$ or $qKx$ to construct a dichotomy. We claim all $H$-invariant dichotomies with quasipolarity $q$ arise in this way, as we now show. If we have an $H$-invariant dichotomy $D\in M_{q}^{H}$, then $qD=\complement D$ and for $h\in H$
\[
(qhq^{-1})D = qh(q^{-1}D)=qh(\complement D) = q\complement D = D,
\]
because $\complement D$ is also invariant under $h$. This means $qHq^{-1}\leq G_{D}$, and since
$K=\langle H, qHq^{-1}\rangle$ (given that any element of $K$ can be written as a product of alternating elements of $H$ and $q$ with an even number of $q$'s, and that $qHq^{-1}$ is a subgroup of both $K_{0}$ and $L$),
this implies $K\leq G_{D}$. On the other hand, a $K$-invariant
dichotomy is a fortiori $H$-invariant, and the conclusion follows.
\end{proof}

Using the Lemma \ref{L:Orbits} and Remark \ref{R:Empty} we have
\begin{equation}\label{E:BlackSide}
s(2k)=\frac{1}{|G|}\sum_{q\in\mathcal{Q}(G)}\sum_{H\leq K_{0}}\mu(1,H)2^{|S/\langle H,q\rangle|}.
\end{equation}

Note that, by Lemma \ref{L:CaracKernel}, \eqref{E:WhiteSide} can be rewritten as
\begin{equation}\label{E:WhiteSideRedux}
Q_{\mathrm{rig}}(-1) = \frac{1}{|G|}\sum_{H\not\leq K_{0}}\mu(1,H)2^{|S/H|}.
\end{equation}

We can prove the following now that the similarities between \eqref{E:BlackSide} and \eqref{E:WhiteSideRedux} are unmistakable.

\section{The Main Theorem and Its Proof}

\begin{thm}\label{Th:Main} Let $k$ be an odd positive integer, $Q_{\mathrm{rig}}(x)$ the inventory polynomial of rigid bicolour pattern classes over $S=\mathbb{Z}/2k\mathbb{Z}$ and $s(2k)$ the number of strong dichotomy classes over $S$ (i.e., the rigid self-complementary bicolour patterns). The following holds:
\[
Q_{\mathrm{rig}}(-1) = -s(2k).
\]
\end{thm}

\begin{proof}
Consider first
the set 
\[
P=\{L\leq G:L\not\leq K_{0}\}.
\]

Then
\begin{multline*}
\sum_{q\in\mathcal{Q}(G)}\sum_{H\leq K_{0}}\mu(1,H)2^{|S/\langle H,q\rangle|}
=\\
\sum_{L\in P}\left(\sum_{q\in\mathcal{Q}(G)}\sum_{H\leq K_{0}}[L=\langle H,q\rangle]\mu(1,H)\right)2^{|S/L|}
\end{multline*}
so we can define
\[
C(L) = \sum_{q\in\mathcal{Q}(G)}\sum_{H\leq K_{0}}[L=\langle H,q\rangle]\mu(1,H).
\]

We observe that
\begin{align*}
\sum_{J\leq L,J\not\leq K_{0}} C(J) &= \sum_{J\leq L,J\not\leq K_{0}}\sum_{q\in\mathcal{Q}(G)}\sum_{H\leq K_{0}}[J=\langle H,q\rangle]\mu(1,H)\\
&=\sum_{q\in\mathcal{Q}(G)}\sum_{H\leq K_{0}}\mu(1,H)\sum_{J\leq L,J\not\leq K_{0}}[J=\langle H,q\rangle]\\
&=\sum_{q\in\mathcal{Q}(G)}\sum_{H\leq K_{0}}\mu(1,H)[\langle H,q\rangle \leq L]\\
&=\sum_{q\in\mathcal{Q}(G)}\sum_{H\leq K_{0}}[q\in L][H\leq L\cap K_{0}]\mu(1,H)\\
&=\sum_{H\leq K_{0}}\sum_{q\in\mathcal{Q}(G)}[q\in L][H\leq L\cap K_{0}]\mu(1,H)\\
&=|\mathcal{Q}(G)\cap L|\sum_{H\leq K_{0}}[H\leq L\cap K_{0}]\mu(1,H)\\
&=|\mathcal{Q}(G)\cap L|\sum_{H\leq L\cap K_{0}}\mu(1,H)=|\mathcal{Q}(G)\cap L|[L\cap K_{0}=1]
\end{align*}
and we have two cases: if $L\cap K_{0}$ is not trivial, then this cumulative sum is $0$. Otherwise, $L$ is of order $2$,
thus there is only another element $q$ which is its own inverse and it is not in $K_{0}$, which means it has an odd
translational part, and hence it cannot have fixed points. Thus it is a quasipolarity, and $|\mathcal{Q}(G)\cap L|=1$.
In summary,
\[
\sum_{J\leq L,J\not\leq K_{0}} C(J) = [L\cap K_{0}=1].
\]

On another hand,
\begin{align*}
[L=1]=\sum_{J\leq L} \mu(1,J) &= \sum_{J\leq L,J\leq K_{0}}\mu(1,J)+\sum_{J\leq L,J\not\leq K_{0}}\mu(1,J)\\
&= \sum_{J\leq L\cap K_{0}}\mu(1,J)+\sum_{J\leq L,J\not\leq K_{0}}\mu(1,J)\\
&= [L\cap K_{0}=1]+\sum_{J\leq L,J\not\leq K_{0}}\mu(1,J)
\end{align*}
and, whenever $L\not\leq K_{0}$, this implies that $[L=1]=0$ and
\[
-\sum_{J\leq L,J\not\leq K_{0}}\mu(1,J)=[L\cap K_{0}=1] = \sum_{J\leq L,J\not\leq K_{0}} C(J).
\]

Using Möbius inversion (within the restricted poset $P$; note that $\mu$ is the Möbius function of the full subgroup lattice
of $G$), we conclude that
\[
 C(L) = -\mu(1,L).
\]

Substituting this identity in the decomposition of \eqref{E:BlackSide} it follows at once that it equals \eqref{E:WhiteSideRedux}, save for a sign.
\end{proof}

\section{Computations}

A function in GAP 4.11.1 that implements the formula
\[
s(2k)=
-\frac1{|G|}
\sum_{H\not\le K_0}\mu(1,H)2^{|S/H|}
\]
obtained from Theorem 6.1 was coded in order to calculate the values of
$s(2k)$ for $k=27,29,31,33,35,37,39,41,43,45$, extending the ones that appear
in \cite{oA18}. The results appear in Table \ref{T:s2k}. The GAP code in the appendix evaluates this
sum over conjugacy classes of subgroups, multiplying each term by the size of the corresponding
conjugacy class.

\begin{table}[H]
\caption{Computed values of $s(2k)$ for $k=27+2j$, $0\leq j \leq 9$, using the formula following Theorem \ref{Th:Main}, programmed in GAP 4.11.1.}.
\label{T:s2k}
\begin{tabular}{|c|l|} \hline $k$ & $s(2k)$ \\ \hline $27$ & $3864448$ \\ \hline $29$ & $9916395$ \\ \hline $31$ & $36943701$ \\ \hline $33$ & $312102725$ \\ \hline $35$ & $981531823$ \\ \hline $37$ & $1960450765$ \\ \hline $39$ & $16442472485$ \\ \hline $41$ & $28158172173$ \\ \hline $43$ & $107150534181$ \\ \hline $45$ & $977333969800$ \\ \hline \end{tabular}
\end{table}

\section*{Acknowledgments}

The author acknowledges the assistance of OpenAI's ChatGPT 5.5 Pro during the exploratory stage that led to the use of the Möbius identity \eqref{E:IdMobius}, and also to the GAP coding for computations. The author reviewed all outputs and takes full responsibility for the final content.

\section*{Appendix: Source Code}

\begin{verbatim}
AffinePerm := function(n, u, v)
  # Return the permutation of {1,...,n} induced by the
  # affine map e^u.v : x |-> v*x + u on Z/nZ.  The
  # residue x is represented by the point x+1.
    return PermList(List([1..n],
        i -> ((v*(i-1) + u) mod n) + 1));
end;

AffineGroupData := function(n)
    local units, genG, genK0, G, K0;

  # Units of Z/nZ.
    units := Filtered([0..n-1], v -> GcdInt(v,n) = 1);

  # Generators for Aff(Z/nZ): the translation
  # e^1.1 and the multiplications e^0.v, with v
  # a unit.
    genG := Concatenation(
        [AffinePerm(n,1,1)],
        List(units, v -> AffinePerm(n,0,v))
    );

  # Generators for the parity-preserving affine
  # subgroup K0: the even translation e^2.1 and
  # all unit multiplications e^0.v.
    genK0 := Concatenation(
        [AffinePerm(n,2 mod n,1)],
        List(units, v -> AffinePerm(n,0,v))
    );

    G := Group(genG);
    K0 := Group(genK0);

    return rec(G := G, K0 := K0);
end;

BottomMoebiusTom := function(tom)
    local subs, nrs, ords, inds, mu, i, pos, j, s;

  # Compute the bottom Möbius values mu(1,H) in the
  # subgroup lattice, using subgroup-containment
  # data from the table of marks.
  # Here H ranges over representatives of conjugacy
  # classes of subgroups.

    subs := SubsTom(tom);
    nrs  := NrSubsTom(tom);
    ords := OrdersTom(tom);

    inds := [1..Length(ords)];
    SortBy(inds, i -> ords[i]);

    mu := [];

    for i in inds do
        if ords[i] = 1 then
            mu[i] := 1;
        else
            s := 0;

            for pos in [1..Length(subs[i])] do
                j := subs[i][pos];

                if j <> i then
                    s := s + nrs[i][pos] * mu[j];
                fi;
            od;

            mu[i] := -s;
        fi;
    od;

    return mu;
end;

StrongDichotomiesOddByTom := function(k)
    local n, data, G, K0, tom, mu, lengths, total,
          i, H, SH;

    if k mod 2 = 0 then
        Error("This formula is for odd k.");
    fi;

    # Initialisation.
    n := 2*k;
    data := AffineGroupData(n);
    G := data.G;
    K0 := data.K0;

    tom := TableOfMarks(G);
    mu := BottomMoebiusTom(tom);
    lengths := LengthsTom(tom);

    total := 0;

  # Compute
  #
  # s(2k) = -1/|G|sum_{H not <= K0} mu(1,H) 2^{|S/H|}.
  #
  # The sum is taken over conjugacy classes of
  # subgroups.  Therefore each term is multiplied
  # by lengths[i], the size of the conjugacy class
  # of the subgroup representative H.
    for i in [1..Length(lengths)] do
        H := RepresentativeTom(tom, i);

        if not IsSubgroup(K0, H) then
            SH := Length(Orbits(H, [1..n], OnPoints));

            total := total + lengths[i] * mu[i] * 2^SH;
        fi;
    od;

    return - total / Size(G);
end;
\end{verbatim}
\bibliographystyle{amsplain}
\bibliography{moebius}
\end{document}